\documentclass[reqno,12pt]{amsart}
\usepackage{a4wide,color,eucal,enumerate,mathrsfs}
\usepackage[normalem]{ulem}
\usepackage{amsmath,amssymb,epsfig,amsthm} 
\usepackage[latin1]{inputenc}
\usepackage{psfrag}
\usepackage{hyperref,cleveref,tikz}

\numberwithin{equation}{section}

\newtheorem{theorem}{Theorem}[section]
\newtheorem{corollary}[theorem]{Corollary}
\newtheorem{proposition}[theorem]{Proposition}
\newtheorem{lemma}[theorem]{Lemma}
\newtheorem{definition}[theorem]{Definition}

\theoremstyle{remark}
\newtheorem{remark}[theorem]{Remark}

\title[Stable solutions may have nonconvex superlevel sets]{Strictly stable solutions in uniformly convex planar domains may have nonconvex superlevel sets}
\author{Yi Ru-Ya Zhang}
\dedicatory{Dedicated to Xavier Cabr\'e on the occasion of his 60th birthday}
\date{\today}

\address{State Key Laboratory of Mathematical Sciences, Academy of Mathematics and Systems Science, Chinese Academy of Sciences, Beijing 100190, China}
\address{Institute of mathematics, Academy of Mathematics and Systems Science, the Chinese Academy of Sciences, Beijing 100190, China}
\email{yzhang@amss.ac.cn}
 \thanks{The author is funded by the National Key R\&D Program of China (Grant No. 2025YFA1018400 \&  No. 2021YFA1003100), NSFC Grant No. 12288201 \& No. 12571128, the Chinese Academy of Sciences, and CAS Project for Young Scientists in Basic Research, Grant No. YSBR-031. }

\begin{document}

\begin{abstract}
We construct smooth, uniformly convex planar domains that admit minimal, strictly stable solutions of a semilinear Dirichlet problem whose superlevel sets are nonetheless nonconvex. The class of admissible nonlinearities includes, in particular, two prototypical cases: the Gelfand-type nonlinearity $e^u$ and the family of shifted power-type nonlinearities $(a+u)^p$, where $a>0$ and $p>1$. By applying the elementary scaling properties of the Dirichlet problem, we also show that the same lack of convexity of superlevel sets holds for the corresponding parameter-dependent equations. These results provide a negative answer to a question posed by Brezis, who inquired whether the stability of a solution necessarily entails quasiconcavity for these prototypical stable configurations.
\end{abstract}

\maketitle

\section{Introduction}

Let $\Omega\subset \mathbb R^2$ be a bounded smooth domain.  We consider classical solutions of
\begin{equation}\label{eq:intro-dirichlet}
\begin{cases}
-\Delta u=f(u) & \text{in }\Omega,\\
u>0 & \text{in }\Omega,\\
u=0 & \text{on }\partial\Omega,
\end{cases}
\end{equation}
where $f:[0,\infty)\to(0,\infty)$ is a prescribed nonlinearity.  For
$t\in\mathbb R$ we write
$$
        \Omega_t:=\{x\in\Omega:u(x)>t\}
$$
for the superlevel set of $u$.  The function $u$ is called
\emph{quasiconcave} if every superlevel set $\Omega_t$ is convex.

A basic question in the qualitative theory of elliptic equations asks whether convexity of the domain forces convexity of the level sets.  In the present
setting this can be phrased as follows: If $\Omega$ is convex and $f\ge0$, must every positive solution of \eqref{eq:intro-dirichlet} be quasiconcave?  This
question is natural from several viewpoints: Comparison principles, the maximum principle on geometric quantities, rearrangement inequalities, and the
Brunn--Minkowski theory all suggest that convexity of the ambient set should be inherited by distinguished elliptic objects.

Many classical results support this expectation.  Makar-Limanov \cite{M1971} proved the quasiconcavity of the torsion function in convex planar domains.
Brascamp and Lieb later proved log-concavity, hence quasiconcavity, for the first Dirichlet eigenfunction in a convex domain  in \cite{BL1975}.  Lions \cite[Remark 3]{L1981} subsequently formulated the broad conjecture that quasiconcavity should hold for semilinear problems with nonnegative H\"older right-hand side.  For power-type equations of the form $f(u)=\lambda u^p$, $0<p<1$, the corresponding power-concavity was established in the plane by Keady \cite{Kea1985} and in higher dimensions by Kennington \cite{Ken1985}.  There are further positive results, using concavity transformations, constant-rank methods, quasiconcave envelopes, or maximum principles for level-set curvature, including, for instance,  \cite{APP1981,CJ1982,Ko1983,K1983,K1985,K1986,KL1987,GP1993,L1994,ALL1997,CS2003,GG2004,BGMX2011,X2011,MSY2012};
see also Kawohl's monograph \cite{K19852} and the survey perspective in \cite{A2015}.  In particular, when $\Omega$ is a ball, the moving-plane method of Gidas, Ni, and Nirenberg \cite{GNN1979} implies that positive solutions are radial and monotone, and hence quasiconcave.

It is now known, however, that convexity of the domain alone is not sufficient in full generality.  Hamel, Nadirashvili, and Sire constructed smooth convex planar domains and smooth nonlinearities for which the associated solution is not quasiconcave in \cite{HNS2016}.  Their examples show that the unrestricted form of Lions' question is false.  They leave open, however, the more rigid stable regime that arises naturally in the theory of minimal solutions and in Brezis' problem. 
See also the potential nonconvexity of level sets for mean curvature equations \cite{W2014}, the recent results on the breakdown of level set convexity for Robin torsion functions \cite{LWY2025}, and the construction of positive, monotone (and hence stable) solutions to semilinear elliptic PDEs exhibiting non-convex superlevel sets (associated with non-convex nonlinearities) via the incompressible Euler system \cite{GRXX2025}. 

We recall the stability notion.  A solution of \eqref{eq:intro-dirichlet} is
\emph{stable} if
\begin{equation}\label{stable}
\int_\Omega |\nabla\phi|^2\,dx-
\int_\Omega f'(u)\phi^2\,dx\ge0
\qquad\text{for every }\phi\in C_c^1(\Omega).
\end{equation}
It is \emph{strictly stable} if
$$
 \lambda_1(-\Delta-f'(u),\Omega)>0,
$$
where $\lambda_1$ denotes the first Dirichlet eigenvalue.  For the usual parameterized
problem $-\Delta u=\lambda f(u),$ 
the corresponding stability inequality is
obtained by replacing $f'(u)$ in \eqref{stable} with $\lambda f'(u)$.

Brezis  asked in \cite[Open Problem 3]{B2003} whether stability restores convexity of level sets for the model
nonlinearities appearing in the Gelfand problem and in the power problem: 

\smallskip
{\it If $f(u)=e^u$ or $f(u)=(1+u)^p$, and $u$ is a
stable solution in a convex domain, must $u$ be quasiconcave? }
\smallskip

This paper shows that the stable version also fails, even in dimension two and even for minimal strictly stable solutions. From a historical perspective, a seminal result of Cabr\'e and Chanillo \cite{CC1998} provided strong evidence in favor of a positive resolution: In a uniformly convex planar domain, every stable solution with a nonnegative right-hand side possesses a unique critical point, and the Hessian at this point is strictly negative definite. As a consequence, all sufficiently high superlevel sets are smooth, closed, and uniformly convex. The uniqueness of the critical point was subsequently extended to arbitrary convex planar domains by De Regibus, Grossi, and Mukherjee \cite{DGM2021}. Therefore, if a counterexample exists, the failure of convexity must occur at an intermediate level, not near the maximum, and cannot be explained by multiple critical points.

The purpose of this paper is to construct such a negative answer.  We prove that, even in a smooth uniformly convex planar domain, and even for the minimal strictly stable solution, a superlevel set may fail to be convex for the
Gelfand and shifted-power nonlinearities.

\subsection{The one-dimensional fold condition}

The construction is based a one-dimensional saddle-node bifurcation, as noticed by the author during the analysis of Problem 6.1 in Br\'ezis \cite{B2023}.

Let
$$
  f\in C^3([0,\infty)),\qquad   f>0,   \qquad f'\ge0,  \qquad f''\ge0.
$$
For each amplitude $A>0$, let $U_A$ be the even solution of
$$
 -U_A''(\xi)=f(U_A(\xi)),  \qquad U_A(0)=A,   \qquad U_A'(0)=0,
$$
on the maximal interval on which it remains positive.  If $U_A$ first reaches
zero at $\xi=h(A)>0$, then
$$
  U_A(h(A))=0, \qquad U_A(\xi)>0\quad \text{for } \ 0\le \xi<h(A).
$$
Writing
$$
  F(s):=\int_0^s f(\tau)\,d\tau,
$$
the half-width is given by the map
$$
h(A)=\int_0^A\frac{ds}{\sqrt{2(F(A)-F(s))}};
$$
see Lemma~\ref{lem:time-map} below. 

\begin{definition}\label{def:fold-admissible}
Let
$$
  f\in C^3([0,\infty)),\qquad   f>0,   \qquad f'\ge0,  \qquad f''\ge 0.
$$
We say that $f$ is \emph{fold-admissible} if there exists $A_c>0$ such that $h$ is $C^3$ in a neighborhood of $A_c$,
$$
   h'(A_c)=0,  \qquad h''(A_c)<0,
$$
and the one-dimensional lower branch is stable near $A_c$ in the following
sense.  Let $I_A:=(-h(A),h(A))$ and let
$$
 \mathcal L_A:=-\frac{d^2}{d\xi^2}-f'(U_A(\xi))
$$
be the Dirichlet linearized operator on $I_A$.  Then there exists $\delta_0>0$
such that
$$
 \lambda_1(\mathcal L_A,I_A)>0  \qquad\text{for } \ A\in(A_c-\delta_0,A_c),
$$
whereas $ \lambda_1(\mathcal L_{A_c},I_{A_c})=0.$
\end{definition}

Thus, a nonlinearity is fold-admissible if  its lower stable branch terminates at a nondegenerate fold.  This is the only structural hypothesis used in the general theorem below. Similar phenomena have been studied, for instance,  in a series of work by Crandall--Rabinowitz \cite{CR1973,CR1975} for the parameterized problem; however the parameter used here is the amplitude $A$, rather than $\lambda$. 

\subsection{Main theorem}

Throughout, a smooth bounded planar domain is called \emph{uniformly convex} if its boundary curvature is bounded below by a positive constant.
A positive classical solution $u$ of \eqref{eq:intro-dirichlet} is called  \emph{minimal} if $u\le v$ in $\Omega$ for every other positive classical  solution $v$ of the same Dirichlet problem.  Our main result is the following.

\begin{theorem}\label{thm:main}
Let $f$ be fold-admissible.  Then there exist a bounded smooth uniformly convex domain $\Omega\subset\mathbb R^2$, a number $t_0>0$, and the minimal solution $u$ of \eqref{eq:intro-dirichlet} such that $u$ is strictly stable while
$$
        \Omega_{t_0}=\{x\in\Omega:u(x)>t_0\}
$$
is not convex.
\end{theorem}

For the nonlinearities in Brezis' question this gives the following concrete
consequence.

\begin{corollary}\label{cor:model}
Let either $f(s)=e^s$ or $f(s)=(a+s)^p$, where $a>0$ and $p>1$.  Then there
exist a bounded smooth uniformly convex domain $\Omega\subset\mathbb R^2$, a
number $t_0>0$, and a minimal solution $u$ of \eqref{eq:intro-dirichlet} such
that $u$ is strictly stable and $\{u>t_0\}$ is not convex.
\end{corollary}

 The verifications of fold-admissibility for these two families  are in Propositions~\ref{prop:exp-fold} and~\ref{prop:power-fold}.  Together with the
scaling statement in Corollary~\ref{cor:parameter-version}, Corollary~\ref{cor:model} applies to the parameterized equations  $-\Delta u=\lambda e^u$ and $-\Delta u=\lambda(1+u)^p$.

\subsection{Mechanism of the construction}

We next describe in detail the geometric mechanism underlying the proof.  Even though the fundamental ideas share similarities to those in \cite{W2014}, the underlying mechanism is substantially different.

Let $h_c := h(A_c)$. For a fixed level $t\in(0,A_c)$, let $Y_t(A)>0$ be the function determined by
$$
        U_A(Y_t(A))=t.
$$
If $A_-(h)$ denotes the inverse of the lower branch $A\mapsto h(A)$, $\mathcal Y_t(h):=Y_t(A_{-}(h))$ and $Y_c:=Y_t(A_c)$, then the nondegenerate fold gives the expansion
\begin{equation}\label{eq:expansion}
        \mathcal Y_t(A_-(h)) =Y_c-c_t\sqrt{h_c-h}+o\bigl(\sqrt{h_c-h}\bigr)
        \qquad\text{as }h\to h_c^-,
\end{equation}
with $c_t>0$.  The sign and the square-root singularity are essential here.  This expansion tells  that, although the stable profiles themselves vary smoothly along the
lower branch away from $h_c$, the height of a fixed interior level can become extremely singular when expressed as a function of the interval half-width near the fold. 

We now extend this one-dimensional phenomenon to a thin and long domain. Let us introduce coordinates $(X,\eta)$ to represent points in the domain, whose central region, called the \emph{slow channel}, is given by
$$
\{(X,\,\eta)\colon |\eta| < H(X).\}
$$
The function $H$ is assumed to be concave and to attain values slightly below $h_c$. Consequently, the upper boundary of the slow channel is concave. By attaching smoothly varying, uniformly convex caps at the lateral ends, the overall domain becomes uniformly convex.

We write $\mathcal U_h(\eta):=U_{A_-(h)}(\eta)$. 
Then formally, the solution within the channel is represented by the slowly varying one-dimensional profile
$$
 V(X,\eta) := \mathcal U_{H(X)}(\eta).
$$
The upper vertical coordinate of the (formal) level set $\{V = t\}$ is roughly given by
$$
X\longmapsto \mathcal Y_t(H(X)).
$$
If we write $H(X) = h_c - D(X)$, then the leading-order term in \eqref{eq:expansion} becomes $-c_t \sqrt{D(X)}$. Consequently, we choose a positive \emph{convex} function $D$ for which $\sqrt D$ strictly \emph{concave} near the selected midpoint. Since the leading term is $-c_t\sqrt D$, the resulting level-height function becomes strictly midpoint-convex there, and hence violates the concavity required
for the upper height function of a convex superlevel set. This is exactly the main idea by which the effective one-dimensional fold dominates, and ultimately overrides the convexity properties of the ambient domain.

The remaining analytical task is to pass from the formal profile $V$ to the genuine minimal solution $u_\epsilon$ in a very long, slowly varying domain. A barrier argument, together with the maximum principle, shows that, in the central region of the domain, 
$u_\epsilon(X/\epsilon,\eta)$
remains uniformly close to $V(X,\eta)$ as $\epsilon \to 0^+$. Since the formal level set exhibits a finite midpoint violation, the same violation persists for the actual level set of $u_\epsilon$. Strict stability then follows from the positivity of the one-dimensional spectral gap on the lower branch, together with a vertical slicing argument in the channel and a comparison with a suitably chosen stable supersolution.

This mechanism is different from previously known negative examples. For instance, 
Monneau and Shahgholian \cite{MS2005} constructed nonconvex levels in convex rings by passing
through a free-boundary limit \cite{A1989}.  Hamel, Nadirashvili, and Sire \cite{HNS2016} produced counterexamples in convex domains and convex rings for specially constructed (nonconvex) nonlinearities, where stability is not part of their conclusion.  Gladiali--Grossi \cite{GG2022} and De Regibus--Grossi--Mukherjee \cite{DG2022} showed that
strip-like domains can produce many critical points, including for stable solutions in the latter work.  By contrast, our domain is smooth, bounded, simply connected, and uniformly convex; the nonlinearities are
positive, increasing, and convex; and the solution is minimal and strictly stable.  The nonconvexity is therefore not a consequence of topology, non-monotonicity of the reaction term, or multiple critical points, but of the
square-root response created by the stable one-dimensional fold. 

\subsection*{Organization}

Section~\ref{sec:preliminaries} contains the planar geometric preliminaries, the analysis on the one dimensional fold, and the verification of fold-admissibility for the exponential and shifted-power nonlinearities.  In particular we prove the square-root expansion
\eqref{eq:expansion}. 

Section 3 constructs the domains and proves the counterexample. More specifically, in section~\ref{sec:domains} we construct the uniformly convex slow channels.  Section~\ref{sec:solution-stability} proves existence,
minimality, and strict stability of the relevant solution.  The maximum-principle approximation is proved in Section~\ref{sec:c0-approx}, and the midpoint
nonconvexity argument is completed in Section~\ref{sec:nonconvexity}.  The final
part records the parameterized version.

\medskip

\noindent{\bf AI disclaimer}: The core mathematical contributions of this paper were developed independently by the author and were not substantially derived from AI, although AI was used hastily to fill in additional details. Indeed, at that moment, several of Brezis's open problems were being resolved. The author was worried and would like to rescue his idea. 

More precisely, the core idea underlying the mechanism for constructing the counterexample came out during the author's investigation of Open Problem 6.1 in Brezis \cite{B2023} and the reading of the manuscript by Dancer \cite{D1993} on thin domains. During the process, the author realized that the results seemed to be connected to the conceptual approach introduced by Wang \cite{W2014}, in which a counterexample to the mean curvature equation was constructed. Subsequently, the author used AI to analyze the one-dimensional problem with nonlinearity $e^t$, from which a folding phenomenon in the solution set was observed. This numerical observation motivated the definition of fold-admissibility in Definition~\ref{def:fold-admissible} and the function $H$ in Lemma~\ref{lem:formal-midpoint}. Then the author prepared a draft for the counterexample based on this fold-admissibility together with the three-point argument in \cite{HNS2016}, where the idea of using a barrier function, rather than the sub-/super-solution argument in Wang, was also adopted. Some of the lemmas were only with sketched proofs, and AI were asked to give full details. In the end, the author used ChatGPT 5.5 for page-by-page proofreading, to refine the writing and to detect and fix technical errors. Following the AI's suggestions, the author included additional details in Lemma~\ref{lem:support-function}, Lemma~\ref{lem:cap-completion}, Proposition~\ref{prop:exp-fold}, Proposition~\ref{prop:power-fold}, and Lemma~\ref{lem:height-closeness}. 
After that, the author checked everything line by line and rewrote certain parts of the manuscript. A preliminary version was finally completed in early June and circulated to several experts for feedback.

\medskip
\noindent{\bf Acknowledge:} The author would like to express his sincere gratitude to Alessio Figalli for introducing him to this problem in 2018, for many enlightening discussions over the subsequent years, and for his patience and generosity in listening to the author's mistakes, questions, and exploratory ideas. Although the present paper ultimately takes a different direction, the author's understanding of the problem has benefited greatly from these interactions.

\section{Preliminaries}\label{sec:preliminaries}

We first recall  the existence result for solutions  to the parameterized Dirichlet problem. 
This is only used to deal with the parameterized problem. 
\begin{theorem}[\cite{B2003, N2000}]\label{thm:existence}
Let $\Omega\subset\mathbb R^2$ be a bounded smooth domain and fix a nonlinearity $f$ which is positive, increasing, convex and superlinear
$$\frac{f(s)}{s}\to \infty.$$ 
Then there exists a constant $\lambda^*=\lambda^*(\Omega,\,f)\in(0,\,+\infty)$ so that, for every $0<\lambda< \lambda^*$, there exists a classical stable solution to
$$
\left\{
\begin{array}{ll}
-\Delta u=\lambda f(u) & \text{ in } \Omega,\\
u=0 &\text{ on } \partial \Omega.
\end{array} \right.
$$
When $\lambda=\lambda^*$, there is a weak solution $u^*\in L^1(\Omega)$.  When $\lambda>\lambda^*$, there exists no solution.     
\end{theorem}

\subsection{Planar geometry}
We next record some elementary planar facts used later in the construction. 
Let $K\subset\mathbb R^2$ be a set. 
In the physical coordinate space $\mathbb{R}^2$, we denote a   point by $(x_1, x_2)$. For any $\epsilon > 0$, we also introduce the \emph{slow coordinates} on $\mathbb{R}^2$ via the affine transformation
$$(x_1, x_2) \mapsto (X, \eta) \qquad \text{with} \quad X = \epsilon x_1,\ \eta = x_2.$$
In these coordinates, we  represent points in $\mathbb{R}^2$ as $(X,\eta)$, where $X$ denotes the horizontal coordinate and $\eta$ denotes the vertical coordinate. 

Define the horizontal projection of $K$ by
$$
\pi_X(K):=\{X\in\mathbb R:\text{there exists }\eta\in\mathbb R\text{ with }(X,\eta)\in K\}.
$$
When $\pi_X(K)$ is nonempty and each vertical section of $K$ is bounded from above, we define the upper height function
$$
\mathcal H_K(X):=\sup\{\eta\in\mathbb R:(X,\eta)\in K\},
\qquad X\in\pi_X(K).
$$

\begin{lemma}\label{lem:convex-height}
Let $K\subset\mathbb R^2$ be a non-empty convex set.  Assume every vertical section is bounded above. Then $\pi_X(K)$ is an interval and that $\mathcal H_K(X)<\infty$ for every $X\in\pi_X(K)$.  Then
$\mathcal H_K$ is concave on $\pi_X(K)$.
\end{lemma}

\begin{proof}
Let $X_1,X_2\in\pi_X(K)$ and let $\theta\in[0,1]$.  Set
$$
X_\theta:=\theta X_1+(1-\theta)X_2.
$$
We prove that
\begin{equation}\label{eq:concave-H}
    \mathcal H_K(X_\theta)\ge \theta \mathcal H_K(X_1)+(1-\theta)\mathcal H_K(X_2).
\end{equation}

Let $\delta>0$.  By the definition of the supremum, there exist
$\eta_1,\eta_2\in\mathbb R$ such that
$$
(X_1,\eta_1)\in K,\qquad (X_2,\eta_2)\in K,
$$
and
$$
\eta_1>\mathcal H_K(X_1)-\delta,\qquad \eta_2>\mathcal H_K(X_2)-\delta.
$$
Since $K$ is convex,
$$
\theta(X_1,\eta_1)+(1-\theta)(X_2,\eta_2) = (X_\theta,\theta\eta_1+(1-\theta)\eta_2)
$$
belongs to $K$.  Therefore
$$
\mathcal H_K(X_\theta)\ge \theta\eta_1+(1-\theta)\eta_2.
$$
Using the inequalities for $\eta_1$ and $\eta_2$, we obtain
$$
\mathcal H_K(X_\theta)
> \theta \mathcal H_K(X_1)+(1-\theta)\mathcal H_K(X_2)-\delta.
$$
Since $\delta>0$ was arbitrary, the desired concavity inequality \eqref{eq:concave-H} follows.
\end{proof}

We  use support functions  to close the long channel by smooth uniformly convex caps on the sides.   For an angle $\theta\in\mathbb R$, set
$$
 n(\theta):=(\cos\theta,\sin\theta),
 \qquad
 \tau(\theta):=(-\sin\theta,\cos\theta).
$$
Thus $n(\theta)$ is a unit vector and $\tau(\theta)=n'(\theta)$ is obtained by
rotating $n(\theta)$ counterclockwise by $\pi/2$.

Recall that, for a  compact convex set $K\subset\mathbb R^2$, its support function with respect to the origin is the function
$$
 s_K(\theta):=\sup_{P\in K} P\cdot n(\theta),
 \qquad \theta\in\mathbb R,
$$
where $P\cdot n(\theta)$ denotes the Euclidean inner product.  Conversely, a smooth $2\pi$-periodic function $s$ with positive radius of curvature determines a smooth uniformly convex curve, given by the following formula. For  background on support functions of convex bodies, see e.g. \cite[Section 2.5]{Schneider2013}.

\begin{lemma}\label{lem:support-function}
Let $s:\mathbb R\to\mathbb R$ be a $2\pi$-periodic $C^\infty$ function and assume
that
$$
 s(\theta)+s''(\theta)>0 \qquad\text{for every }\theta\in\mathbb R.
$$
Define
$$ \Gamma_s(\theta):=s(\theta)n(\theta)+s'(\theta)\tau(\theta),
 \qquad \theta\in\mathbb R.
$$
Then $\Gamma_s$ parametrizes a smooth uniformly convex closed curve.  Moreover,
$$
\Gamma_s'(\theta)=\bigl(s(\theta)+s''(\theta)\bigr)\tau(\theta),
$$
and the quantity
$$
 \rho(\theta):=s(\theta)+s''(\theta)
$$
is the radius of curvature at $\Gamma_s(\theta)$.  

Conversely, every smooth uniformly convex closed plane curve admits such a parametrization after choosing an origin and using its support function as a function of the outer normal angle.
\end{lemma}

\begin{proof}
Differentiating $\Gamma_s=s n+s'\tau$, and using $n'=\tau$ and $\tau'=-n$,
we obtain that
$$
 \Gamma_s' =s'n+s\tau+s''\tau-s'n  =\bigl(s+s''\bigr)\tau.
$$
Thus $\Gamma_s'$ never vanishes if $s+s''>0$.  The tangent direction is denoted by $\tau(\theta)$. Since $\theta$ increases from $0$ to $2\pi$, this tangent direction winds once monotonically around the unit circle.  The signed radius of curvature is therefore $\rho=s+s''$, which is positive by assumption.  Hence the curve is uniformly convex as the compactness of the boundary gives
$$\sup_\theta \rho<\infty. $$

For the converse, let a smooth uniformly convex curve be parametrized by its outer unit normal $n(\theta)$. Then its support function can be written as
\begin{equation}\label{eq:support-s}
     s(\theta)=\Gamma(\theta)\cdot n(\theta).
\end{equation}
Since $\Gamma'(\theta)$ is tangent to the curve, differentiating both sides of \eqref{eq:support-s} gives $$\Gamma(\theta)\cdot\tau(\theta)=s'(\theta).$$  
Therefore, combining with 
$$
 \Gamma(\theta)=s(\theta)n(\theta)+s'(\theta)\tau(\theta),
$$
the identity for $\Gamma'$ then identifies $s+s''$ with the radius of curvature.
\end{proof}

We also record the following consequence of the support-function description.

\begin{lemma}\label{lem:cap-completion}
Let $H\in C^\infty(J)$, where $J\subset\mathbb R$ is an open interval, and assume
$$
H(X)>0,\qquad H''(X)<0
\qquad\text{for every }X\in J.
$$
Let $J_0=[a_0,b_0]\subset\subset J$ be a compact subinterval, and let $h_*>0$ satisfy
$$
\sup_{X\in J_0}H(X)<h_*.
$$
Then there exists a bounded $C^\infty$ uniformly convex domain
$\mathcal O\subset\mathbb R^2$ such that
\begin{equation}\label{eq:cap-strip}
\overline{\mathcal O}\subset\{(X,\eta):|\eta|<h_*\},
\end{equation}
and
\begin{equation}\label{eq:exact-central-channel}
\mathcal O\cap (J_0\times\mathbb R)
= \{(X,\eta):X\in J_0,\ |\eta|<H(X)\}.
\end{equation}
Moreover, there exists an open interval $J_0'$ with 
$$J_0\subset\subset J_0'\subset\subset J$$
such that
$$
\partial\mathcal O \cap (J_0'\times \mathbb R)= \{(X,H(X)):X\in J_0'\} \cup \{(X,-H(X)):X\in J_0'\}.
$$
\end{lemma}

\begin{proof}
Choose a compact interval $J_1\subset\subset J$ such that
$$
J_0\subset\subset J_1, \qquad \sup_{X\in J_1}H(X)<h_*.
$$
The upper graph
$$
\gamma_+(X):=(X,H(X)),\qquad X\in J_1,
$$
has curvature
$$
\kappa_+(X)=\frac{-H''(X)}{(1+H'(X)^2)^{3/2}}>0
$$
when it is oriented as the upper boundary of the region below it. Similarly, the lower graph
$$
\gamma_-(X):=(X,-H(X)),\qquad X\in J_1,
$$
has positive curvature when it is oriented as the lower boundary of the region above it.

We now use the support-function extension for strictly convex arcs. Namely, finitely many pairwise disjoint $C^\infty$ strictly convex arcs, prescribed on compact subintervals of the normal-angle circle and satisfying the natural cyclic orientation compatibility, can be extended to a $C^\infty$ closed strictly convex curve. In support-function notation, this amounts to extending the local support functions to a $2\pi$-periodic $C^\infty$ function $s$ with
$$
s+s''>0.
$$
The resulting curve is
$$
\Gamma(\theta)=s(\theta)n(\theta)+s'(\theta)\tau(\theta),
$$
and has positive curvature. See the support-function formalism recalled above and \cite[Section 2.5]{Schneider2013}.

Applying this extension to the two graph arcs over $J_1$, and choosing the two complementary cap arcs inside the horizontal strip $|\eta|<h_*$, we obtain a $C^\infty$ closed curve with positive curvature. The selection of the caps within the strip is possible since the prescribed arcs over $J_1$ are located at a positive distance from the two lines $\eta = \pm h_*$. More precisely, we are given finitely many smooth, strictly convex arcs prescribed on pairwise disjoint compact intervals of normal angles, with a compatible cyclic ordering. 
By possibly shrinking the prescribed neighborhoods, the corresponding local support functions can then be extended to a smooth $2\pi$-periodic support function $s$ satisfying
$$
s + s'' > 0.
$$
Furthermore, since the prescribed arcs lie at a positive distance inside a horizontal strip, the complementary arcs can also be chosen to lie entirely within this strip.

Let $\mathcal O$ be the bounded component enclosed by this curve. Then $\mathcal O$ is a bounded $C^\infty$ uniformly convex domain, \eqref{eq:cap-strip} holds, and the boundary agrees with the graphs $\eta=\pm H(X)$ over a neighborhood $J_0'$ of $J_0$.

It remains only to prove the identity \eqref{eq:exact-central-channel}. Fix $X\in J_0$. Since $\overline{\mathcal O}$ is convex, the vertical section
$$
\mathcal O_X:=\{\eta\in\mathbb R:(X,\eta)\in\mathcal O\}
$$
is an open interval. The two points
$$
(X,H(X)),\quad (X,-H(X))
$$
belong to $\partial\mathcal O$. Since $\overline{\mathcal O}$ is strictly convex, the open segment joining these two boundary points is contained in $\mathcal O$. Hence
$$
(-H(X),H(X))\subset \mathcal O_X.
$$

We claim that no point of $\mathcal O_X$ can lie above $H(X)$. Indeed, if some $\eta_1>H(X)$ belonged to $\mathcal O_X$, then, choosing any $\eta_0\in(-H(X),H(X))$, the point $(X,H(X))$ would be a strict convex combination of the two interior points $(X,\eta_0)$ and $(X,\eta_1)$. Since $\mathcal O$ is open and convex, this would imply $(X,H(X))\in\mathcal O$, contradicting $(X,H(X))\in\partial\mathcal O$. Therefore
$\mathcal O_X\subset (-\infty,H(X)).$
The same argument applied to the lower boundary point $(X,-H(X))$ gives $\mathcal O_X\subset (-H(X),\infty).$
Thus
$$
\mathcal O_X=(-H(X),H(X)).
$$
Since $X\in J_0$ was arbitrary, \eqref{eq:exact-central-channel} follows.
\end{proof}

\subsection{One-dimensional fold}\label{sec:fold-theory}

This subsection develops the one-dimensional results that will be used throughout the construction. All functions considered herein are real-valued, and primes indicate differentiation with respect to the relevant real variable. The spatial variable on the interval is denoted by $\xi$.

Let $f\in C^3([0,\infty))$ satisfy $f(s)>0$ for every $s\ge0$, and recall 
$$
F(s):=\int_0^s f(\tau)\,d\tau, \qquad s\ge 0.
$$
For $A>0$, let $U_A>0$ be the solution of
$$
-U_A''=f(U_A),\qquad U_A(0)=A,\qquad U_A'(0)=0,
$$
on the maximal interval on which $U_A$ is positive. This is well-defined since $$f(U_A)\ge \min_{[0,A]} f>0,$$
and then $U_A$ reaches zero in finite time.  

\begin{lemma} \label{lem:time-map}
For each $A>0$, the function $U_A$ is strictly decreasing on every interval $(0,\xi)$ on which it is positive.  Let $h(A)>0$ be the first positive zero of $U_A$. Then
\begin{equation}\label{eq:FU}
 \frac 1 2\bigl(U_A'(\xi)\bigr)^2=F(A)-F(U_A(\xi)),
\qquad 0\le \xi\le h(A),   
\end{equation}
and
$$
h(A)=\int_0^A\frac{ds}{\sqrt{2(F(A)-F(s))}}.
$$
\end{lemma}

\begin{proof}
Since $f>0$, the differential equation gives
$$
U_A''(\xi)=-f(U_A(\xi))<0
$$
whenever $U_A(\xi)>0$.  As $U_A'(0)=0$, it follows that
$U_A'(\xi)<0$ for every $\xi>0$ for which $U_A$ remains positive.  Thus $U_A$ is strictly decreasing there.

Multiplying the equation $-U_A''=f(U_A)$ by $U_A'$, we obtain that
$$
-U_A''U_A'=f(U_A)U_A'.
$$
Namely, 
$$
-\frac{d}{d\xi}\left(\frac12(U_A')^2\right) = \frac{d}{d\xi}F(U_A).
$$
Hence, 
$$
\frac 1 2 (U_A')^2+F(U_A)
$$
is constant in $\xi$.  Evaluating at $\xi=0$, where $U_A(0)=A$ and $U_A'(0)=0$, yields
$$
\frac12(U_A'(\xi))^2+F(U_A(\xi))=F(A).
$$
This is exactly \eqref{eq:FU}. 

For $0<\xi<h(A)$, the derivative $U_A'(\xi)$ is negative.  Therefore
$$
U_A'(\xi)=-\sqrt{2(F(A)-F(U_A(\xi)))}.
$$
Using the change of variables $s=U_A(\xi)$, and noting that $s$ decreases from
$A$ to $0$ as $\xi$ increases from $0$ to $h(A)$, we conclude that
$$
h(A)=\int_0^{h(A)}d\xi
= \int_A^0\frac{ds}{-\sqrt{2(F(A)-F(s))}}
= \int_0^A\frac{ds}{\sqrt{2(F(A)-F(s))}}.
$$
\end{proof}
\begin{remark}
Observe that, since $U_A(h(A))=0$, it follows from \eqref{eq:FU} that
$$
U_A'(h(A))=-\sqrt{2F(A)}<0.
$$
Thus the first zero is transversal. By uniqueness of the initial-value problem, $U_A$ is even, and we regard it as defined on the symmetric interval
$$
I_A:=(-h(A),h(A)).
$$
Moreover, after extending $f$ arbitrarily as a $C^3$ function to a small interval $(-\delta,\infty)$, the standard smooth dependence of ODE solutions on initial data and the implicit function theorem applied to
$U_A(h(A))=0$
show that $h$ is $C^3$ on compact subintervals of $(0,\infty)$. The values of $h(A)$ and all identities up to the first zero are independent of the chosen extension.
\end{remark}

Recall that $f\in C^3([0,\infty))$.  For an amplitude $A$ for which
$h(A)$ is differentiable, set
$$
I_A:=(-h(A),h(A))
$$
and define the one-dimensional linearized Dirichlet operator
$$
\mathcal L_A:=-\frac{d^2}{d\xi^2}-f'(U_A(\xi))
\quad\text{on }I_A.
$$
Recall that the first Dirichlet eigenvalue of $\mathcal L_A$ is denoted by $\lambda_1(\mathcal L_A,I_A)$.

Let
$$
\omega_A(\xi):=\frac{\partial U_A(\xi)}{\partial A}
$$
where the derivative is taken for fixed $\xi$. Note that, the smooth dependence of
ordinary differential equations on initial data gives $\omega_A\in C^2$ on compact subintervals of $I_A$, and differentiating the equation for $U_A$
with respect to $A$ gives
\begin{equation}\label{eq:eigen-omega}
\mathcal L_ A \omega_A= -\omega_A''-f'(U_A)\omega_A=0 , \qquad \omega_A(0)=1, \qquad \omega_A'(0)=0.
\end{equation}
In particular, $\omega_A$ is not identically zero.

\begin{lemma} \label{lem:zero-eigenvalue}
Assume that $h$ is differentiable at $A$.  Then
\begin{equation}\label{eq:omega}
    \omega_A(h(A))=-U_A'(h(A))h'(A).
\end{equation}
Moreover, when  $h'(A)=0$, one has that $0$ is a Dirichlet eigenvalue of $\mathcal L_A$ on $I_A$.  

On the other hand, if $\lambda_1(\mathcal L_A,I_A)=0,$
then $h'(A)=0$.
\end{lemma}

\begin{proof}
The definition of $h$ yields $U_A(h(A))=0.$
Differentiating it with respect to $A$, we arrive at
$$
\omega_A(h(A))+U_A'(h(A))h'(A)=0.
$$
This proves \eqref{eq:omega}.

If $h'(A)=0$, then $\omega_A(h(A))=0$.  The function $U_A$ is even, hence
$\omega_A$ is even, and so $\omega_A(-h(A))=0$ as well.  Thus, according to \eqref{eq:eigen-omega}, $\omega_A$ is a
nontrivial Dirichlet solution of
$$
\mathcal L_A \omega_A= -\omega_A''-f'(U_A)\omega_A=0
$$
on $I_A$ with zero Dirichlet boundary value.  Therefore $0$ is a Dirichlet eigenvalue.

Now suppose that $\lambda_1(\mathcal L_A,I_A)=0$.  Observe that the first eigenfunction $\phi$ is even in $I_A$ as  the first eigenspace is simple and the potential $f'(U_A)$ is even.  Hence it satisfies a Neumann condition at the origin.  Then by normalizing $\phi$  via $\phi(0)=1$,  $\phi$  and $\omega_A$ solve the same second-order linear equation, and both have zero derivative at $0$.  Thus, uniqueness of the initial value problem of ODE \eqref{eq:eigen-omega} gives that  $\phi$ coincides with $\omega_A$.  Therefore $\omega_A(h(A))=0$.  Since $U_A'(h(A))<0$, \eqref{eq:omega} yields $h'(A)=0$.
\end{proof}

\begin{lemma} \label{lem:lower-branch-stability}
Let $(0,A_*)$ be an interval of amplitudes on which $h$ is $C^1$, $h$ extends $C^1$ to $A_*$ and
$$
h'(A)>0\qquad\text{for } \ 0<A<A_*.
$$
Assume also that $h(A)\to0$ as $A\to 0^+$, and that $f'$ is bounded on
$[0,A_*]$.  Then
$$
\lambda_1(\mathcal L_A,I_A)>0
\qquad\text{for every }0<A<A_*.
$$
If, in addition, $h'(A_*)=0$ and $h'$ has no other zero in $(0,A_*]$, then
$$
\lambda_1(\mathcal L_{A_*},I_{A_*})=0.
$$
\end{lemma}

\begin{proof}
By our assumption on $h$, for small $A>0$, the interval length $2h(A)$ is small.  Let 
$M_*:=\sup_{0\le s\le A_*}f'(s)$.  
By the Poincar\'e inequality on
$(-h(A),h(A))$, for every $\varphi\in W^{1,2}_0(I_A)$, we have
$$
\int_{I_A}|\varphi'|^2\,d\xi
\ge \frac{\pi^2}{4h(A)^2}\int_{I_A}\varphi^2\,d\xi.
$$
Hence, 
$$
\int_{I_A}\left(|\varphi'|^2-f'(U_A)\varphi^2\right)d\xi
\ge \left(\frac{\pi^2}{4h(A)^2}-M_*\right) \int_{I_A}\varphi^2\,d\xi.
$$
Since $h(A)\to 0$, the right-hand coefficient is positive for all sufficiently small $A$.  Thus $\lambda_1(\mathcal L_A,I_A)>0$ for small $A$.

Notice that, the first eigenvalue depends continuously on $A$ as long as the branch is smooth. This follows, for example, by pulling the operators back to a fixed interval and using the Rayleigh quotient characterization of the first eigenvalue.  Now if $\lambda_1$ became zero at some
$A_0\in(0,A_*)$, Lemma~\ref{lem:zero-eigenvalue} would give $h'(A_0)=0$, 
contrary to the assumption $h'>0$ on $(0,A_*)$.  Therefore
$\lambda_1>0$ throughout $(0,A_*)$, and hence
$$\lambda_1(\mathcal L_{A_*}, I_{A_*})= \lim_{A\to (A_*)^-}\lambda_1(\mathcal L_{A},I_{A})\ge 0.$$

Now if $h'(A_*)=0$, Lemma~\ref{lem:zero-eigenvalue} shows that $0$ is a Dirichlet eigenvalue of $\mathcal L_{A_*}$, which implies
$$\lambda_1(\mathcal L_{A_*}, I_{A_*})\le 0.$$
Thus, we conclude
$\lambda_1(\mathcal L_{A_*},I_{A_*})=0$.
\end{proof}

We now derive the basic singularity at a nondegenerate fold, which is the key toward our counterexample construction.  Let $f$ be fold-admissible in the sense stated in the introduction.  Let
$ h_c:=h(A_c), $ and choose  $0<t<A_c.$
For every $A$ sufficiently close to $A_c$, with $A>t$, define
$Y_t(A)\in(0,h(A))$ by
$$
U_A(Y_t(A))=t.
$$
This point is well-defined and unique since $U_A$ is strictly decreasing on
$(0,h(A))$.

\begin{proposition} \label{prop:sqrt-response}
Let $f$ be fold-admissible.  Then there exists $\delta_h>0$ so that, after restricting to the stable branch
$A\in(A_c-\delta_h, A_c)$ sufficiently close to $A_c$, the   map $A\mapsto h(A)$ is
invertible.  Let $A_-(h)$ denote its inverse, and define
$$
\mathcal Y_t(h):=Y_t(A_-(h)).
$$
Then
$$
\mathcal Y_t(h)=Y_c-c_t\sqrt{h_c-h}+o\bigl(\sqrt{h_c-h}\bigr)
\quad\text{as }h\to (h_c)^-,
$$
where $Y_c:=Y_t(A_c)$ and $c_t>0.$
\end{proposition}

\begin{proof}
As $f$ is fold-admissible, then $h'(A_c)=0$ and the inequality $h''(A_c)<0$. They imply that, for
$A<A_c$ sufficiently close to $A_c$, one has $
h'(A)>0.$
Thus $A\mapsto h(A)$ is invertible on the lower branch.

Recall that 
$$
\omega_{A_c}(\xi):=\left.\frac{\partial U_A(\xi)}{\partial A}\right|_{A=A_c}.
$$
Since $h'(A_c)=0$, Lemma~\ref{lem:zero-eigenvalue} shows that $\omega_{A_c}$ is a
Dirichlet eigenfunction corresponding to the eigenvalue $0$.  As $f$ is
fold-admissible, it must be the first eigenfunction. Hence
$$
\omega_{A_c}(\xi)>0
\qquad\text{for }\ -h_c<\xi<h_c.
$$

Due to $U'_A(Y_t(A))<0$, the implicit function theorem yields that $ A\mapsto Y_t(A)$ is $C^1$ near $A_c$.
Now differentiating the identity
$U_A(Y_t(A))=t$
with respect to $A$ gives
$$
\frac{\partial U_A}{\partial A}(Y_t(A))
+ U_A'(Y_t(A))Y_t'(A)=0.
$$
At $A=A_c$, this yields
$$
Y_t'(A_c)=
\frac{\omega_{A_c}(Y_c)}{-U_{A_c}'(Y_c)}.
$$
Notice that, since $\omega_{A_c}(Y_c)>0$ and
$U_{A_c}'(Y_c)<0$, both numerator and denominator in the formula above are positive. Therefore we conclude
\begin{equation}\label{eq:alpha-t}
    \alpha_t:=Y_t'(A_c)>0.
\end{equation}

Since $h'(A_c)=0$ and $h''(A_c)<0$, the Taylor expansion yields, by setting $\delta:=A_c-A$,
$$
h(A)=h_c-\beta\delta^2+o(\delta^2),\qquad \beta:=-\frac12h''(A_c)>0.
$$
Then it follows that 
\begin{equation}\label{eq:expand-h}
    \sqrt{h_c-h(A)}=\sqrt\beta\,\delta+o(\delta).
\end{equation}
Additionally, according to \eqref{eq:alpha-t}
$$
Y_t(A)=Y_c-\alpha_t\delta+o(\delta).
$$
Substituting \eqref{eq:expand-h} into the expansion of $Y_t(A)$ above gives
$$
\mathcal Y_t(h)=Y_c-c_t\sqrt{h_c-h}+o\bigl(\sqrt{h_c-h}\bigr),
\qquad
c_t:=\alpha_t\beta^{-1/2}>0.
$$
This concludes the proposition. 
\end{proof}

Now when a half-width $h$ belong  to a stable lower-branch half-width interval on which $A\mapsto h(A)$ is invertible, we define
$$
   \mathcal U_h(\xi):=U_{A_-(h)}(\xi),\qquad -h\le \xi\le h.
$$
Thus $\mathcal U_h$ is indexed by the half-width $h$, whereas $U_A$ remains indexed by the amplitude $A$.

In the end of this subsection, we verify the model nonlinearities. 

\begin{proposition}\label{prop:exp-fold}
The nonlinearity $f(s)=e^s$ is fold-admissible.
\end{proposition}

\begin{proof}
Recall that
$$\sinh(t)=\frac{e^t-e^{-t}}{2},\quad \cosh(t)=\frac{e^t+e^{-t}}{2}. $$
For $f(s)=e^s$, the one-dimensional solutions can be computed explicitly.  Let $\alpha>0$
and define
$$
U(\xi)=A-2\log\cosh(\alpha \xi),
\qquad
2\alpha^2=e^A.
$$
Then
$$
U''(\xi)=-2\alpha^2\operatorname{sech}^2(\alpha \xi)
$$
and
$$
e^{U(\xi)}=e^A\operatorname{sech}^2(\alpha \xi)
=2\alpha^2\operatorname{sech}^2(\alpha \xi).
$$
Thus $-U''=e^U$, and the boundary condition $U(h)=0$ is equivalent to $e^{A/2}=\cosh(\alpha h).$

Set $\mu:=\alpha h.$
Then
$$
A(\mu)=2\log\cosh\mu,
$$
and since $\alpha=e^{A/2}/\sqrt2=\cosh\mu/\sqrt2$,
\begin{equation}\label{eq:exp-tildeh}
    \tilde h(\mu):=h(A(\mu))=\frac{\mu}{\alpha}=\sqrt2\,\mu\,\operatorname{sech}\mu.
\end{equation}

The map $\mu\mapsto A(\mu)$ is strictly increasing for $\mu>0$, since
$$
A'(\mu)=2\tanh\mu>0.
$$
Therefore critical points of $h$ as a function of $A$ are exactly critical
points of $h$ as a function of $\mu$.

Differentiating \eqref{eq:exp-tildeh},
$$
\tilde h'(\mu)=\sqrt2\,\operatorname{sech}\mu\,(1-\mu\tanh\mu).
$$
Note that the function $\mu\mapsto \mu\tanh\mu$ is strictly increasing from $0$ to $+\infty$.  Hence, there is a unique $\mu_c>0$ satisfying
$$
\mu_c\tanh\mu_c=1.
$$
At this point,
$$
\frac{d^2\tilde h}{d\mu^2}(\mu_c)
=\sqrt2\,\operatorname{sech}\mu_c
\bigl(-\tanh\mu_c-\mu_c\operatorname{sech}^2\mu_c\bigr)<0.
$$
Since $A'(\mu_c)>0$, it follows from the chain rule that
$$
h'(A_c)=0,
\qquad
h''(A_c)= \frac{\tilde h''(\mu_c)}{(A'(\mu_c)^2)}<0,
$$
where $A_c=A(\mu_c)$.  The  formulas above also show that $h$ is smooth
near $A_c$.  Moreover, $h'(A)>0$ on the lower branch
$0<A<A_c$, and $h(A)\to0$ as $A\to 0^+$.

For small $A$, the interval length $2h(A)$ tends to zero, while
$f'(U_A)=e^{U_A}$ remains bounded.  Hence, the first eigenvalue of the
linearized operator is positive for small $A$ by Poincar\'e inequality.  By
Lemma~\ref{lem:lower-branch-stability}, the first eigenvalue remains positive
throughout the lower branch and becomes zero at the unique fold.  Thus
$e^s$ is fold-admissible.
\end{proof}

\begin{proposition} \label{prop:power-fold}
Let $a>0$ and $p>1$.  The nonlinearity
$f(s)=(a+s)^p$ is fold-admissible.
\end{proposition}

\begin{proof}
Let
$\alpha:=\frac{p-1}{2}.$
The primitive of $f$ is
$$
F(s)=\frac{(a+s)^{p+1}-a^{p+1}}{p+1}.
$$
By Lemma~\ref{lem:time-map},
$$
h(A)=\sqrt{\frac{p+1}{2}}
\int_0^A
\frac{ds}{\sqrt{(a+A)^{p+1}-(a+s)^{p+1}}}.
$$
Set
$$
x:=\frac{a}{a+A}\in(0,1).
$$
Using the change of variables $r=\frac{a+s}{a+A}$, we get
$$
h(A)=C_{p,a}\,x^\alpha J(x),
$$
where 
$$C_{p,a}=\sqrt{\frac {p+1}2} a^{\frac {1-p}2}>0$$
depending only on $p$ and $a$, and
$$
J(x):=\int_x^1\frac{dr}{\sqrt{1-r^{p+1}}}.
$$
Let
$$
\Phi(x):=\frac{1}{\sqrt{1-x^{p+1}}}.
$$
Then $J'(x)=-\Phi(x)$.  Differentiating $x^\alpha J(x)$, we find that the critical point satisfies the following equation: 
$$
\alpha J(x)=x\Phi(x).
$$
Set
$$
R(x):=\frac{x\Phi(x)}{J(x)},
\qquad 0<x<1.
$$
Applying the logarithm to both sides of the equation and subsequently differentiating yields
$$
\frac{R'(x)}{R(x)}
=\frac{1}{x}+\frac{\Phi'(x)}{\Phi(x)}+\frac{\Phi(x)}{J(x)}>0.
$$
Additionally, by the definition of $\Phi$ and $J$,
$$
R(x)\to0\quad\text{as }x \to 0^+,
\qquad
R(x)\ge \frac{c_0}{x-1}\to+\infty\quad\text{as }x\to 1^- \quad \text{for some $c_0>0$}.
$$
Therefore the equation $R(x)=\alpha$ has a unique solution $x_c\in(0,1)$.
The critical point is nondegenerate.  Indeed, observe that $g(x):=x^\alpha J(x)$ satisfies
$$g'(x)= x^{\alpha-1} J(x)(\alpha-R(x))$$
and at $x=x_c$, $R(x_c)=\alpha$, therefore $g'(x_c)=0$ and
$$g''(x_c)= -x_c^{\alpha-1} J(x_c)R'(x_c)<0,$$
i.e. $g$ has a strict nondegenerate maximum at $x_c$.  Since $x=a/(a+A)$ is a smooth strictly decreasing
function of $A$, $h(A)$  has a strict nondegenerate maximum at the corresponding amplitude
$$
A_c=a\left(\frac1{x_c}-1\right).
$$
Thus
$$
h'(A_c)=0,
\qquad
h''(A_c)<0.
$$
The representation above also shows that $h$ is smooth near $A_c$.  The sign
on the lower branch is as follows.  The critical point of $x\mapsto x^\alpha
J(x)$ occurs at $x_c$, and this function is increasing for $0<x<x_c$ and
decreasing for $x_c<x<1$.  Since $x(A)=a/(a+A)$ is strictly decreasing with respect to $A$, the
half-width $h(A)$ is increasing for $0<A<A_c$.  Thus $h'(A)>0$ on the lower
branch.  Finally, $h(A)>0$ and $h(A)\to0$ as $A\to 0^+$.

For small $A$, the interval length $2h(A)$ tends to zero, while
$f'(U_A)=p(a+U_A)^{p-1}$ remains bounded.  Therefore the first eigenvalue
of $\mathcal L_A$ is positive for small $A$.  Since $h'$ has no zero on
$(0,A_c)$, Lemma~\ref{lem:lower-branch-stability} implies that the first
eigenvalue is positive throughout the lower branch and becomes zero at
$A_c$.  Hence $(a+s)^p$ is fold-admissible.
\end{proof}

\subsection{The one-dimensional profile}\label{sec:slow-width}

In this subsection we construct, at the level of the formal slowly varying one-dimensional profile, a level height that violates midpoint concavity.  This is the geometric core of the counterexample.  

Let $f$ be fold-admissible.  We use the notation from \Cref{sec:fold-theory}.  In particular,
$$
h_c:=h(A_c).
$$
Choose a level
$$
0<t<A_c.
$$
Up to decreasing the lower-branch half-width neighborhood of $h_c$ if necessary, there exists $\rho_h>0$ so that the lower-branch inverse $A_-(h)$ is defined for $h\in(h_c -\rho_h, h_c)$.  Then by   $0\le \delta_{\rm br} < \rho_h$ such that, for every
$$
h\in(h_c-\delta_{\rm br},h_c),
$$
the square-root expansion of
Proposition~\ref{prop:sqrt-response} applies, and
\begin{equation}\label{eq:level-below-amplitude}
A_-(h)>t.
\end{equation}
For such $h$, let $\mathcal Y_t(h)$ be the upper height of the level $t$ in the
stable one-dimensional profile of half-width $h$, namely
$$
\mathcal Y_t(h)=Y_t(A_-(h)).
$$
By Proposition~\ref{prop:sqrt-response},
$$
\mathcal Y_t(h)=Y_c-c_t\sqrt{h_c-h}+o\bigl(\sqrt{h_c-h}\bigr)
\quad\text{as }h\to (h_c)^-,
$$
with $c_t>0$.

\begin{lemma}\label{lem:sloping-gap}
There exist a number $\ell_0>0$ and   $D_0\in C^\infty(-2\ell_0,2\ell_0)$ so that, $D_0\ge \frac 1 2 $, 
$$
D_0''(X)\ge \frac 1 2>0\qquad \text{for every } X\in(-2\ell_0,2\ell_0),
$$
and for any $0<\ell<\ell_0$ with
$$
X_-:=-\ell, \qquad X_0:=0, \qquad X_+:=\ell, \ 
$$
one has
$$
\sqrt{D_0(X_0)}> \frac{\sqrt{D_0(X_-)}+\sqrt{D_0(X_+)}}{2}.
$$
\end{lemma}

\begin{proof}
Let, for instance,
$$
D_0(X)=1+2X+\frac 1 2 X^2.
$$
Then $D_0''(X)=1>0$.  Also $D_0(0)=1$ and $D_0'(0)=2$.  For
$G(X):=\sqrt{D_0(X)}$, a direct calculation gives
$$
G''(0)=\frac{2D_0(0)D_0''(0)-D_0'(0)^2}{4D_0(0)^{3/2}}
=\frac{2-4}{4}<0.
$$
Thus $G$ is strictly concave in a neighborhood of $0$.  Choose $\ell_0>0$ small enough that $D_0\ge \frac 1 2$ on $(-2\ell_0,2\ell_0)$, and for any $0<\ell<\ell_0$
$$
G(0)>\frac{G(-\ell)+G(\ell)}2.
$$
This yields the desired midpoint inequality.
\end{proof}

\begin{lemma}\label{lem:formal-midpoint}
For $h_c:=h(A_c)$ defined in \Cref{sec:fold-theory}, there exist an open interval $J\subset\mathbb R$, a smooth function
$H:J\to(0,h_c)$, and three points
$$
X_-<X_0<X_+, \qquad X_0=\frac{X_-+X_+}{2},
$$
with the following properties.
\begin{enumerate}[(i)]
    \item $ H''(X)<0$ for every $X\in J.$
    \item For every $X\in J$, the half-width $H(X)$ belongs to the stable lower branch half-width interval.
    \item If $ Y(X):=\mathcal Y_t(H(X)),$
then
$$
\frac{Y(X_-)+Y(X_+)}2>Y(X_0).
$$
\end{enumerate}
Moreover, $H$ may be chosen so that
$$
\sup_{X\in J}H(X)<h_c.
$$
\end{lemma}

\begin{proof}
Let $D_0$, $\ell_0$ be given by
Lemma~\ref{lem:sloping-gap}. Choose  $0<\ell <\ell_0$ and set 
$$X_-=-\ell \qquad X_0=0 \qquad X_+=\ell$$
as in the lemma.  For a small parameter $\delta>0$, set
$$
D_\delta(X):=\delta D_0(X),
\qquad
H_\delta(X):=h_c-D_\delta(X),
\qquad
X\in(-2\ell_0,2\ell_0).
$$
Then $0<H_\delta(X)<h_c$ as $D_0>0$ is bounded. 
Moreover, since $D_0>0$ and $D_0''>0$, we have
$$
H_\delta(X)<h_c, \qquad H_\delta''(X)=-\delta D_0''(X)<0.
$$
This gives (i). 

For $\delta>0$ sufficiently small, all values of $H_\delta$ lie in the stable lower-branch half-width interval on which  Proposition~\ref{prop:sqrt-response} applies. This gives (ii).

Set $Y_\delta(X):=\mathcal Y_t(H_\delta(X)).$
At the three fixed points $X_-,X_0,X_+$, the square-root expansion in Proposition~\ref{prop:sqrt-response} gives
$$
Y_\delta(X_j) = Y_c-c_t\sqrt{\delta}\,\sqrt{D_0(X_j)}+o(\sqrt\delta),
\qquad j\in\{-,0,+\},
$$
as $\delta\to 0^+$.  Therefore
\begin{multline*}
\frac{Y_\delta(X_-)+Y_\delta(X_+)}2-Y_\delta(X_0)
=  \\
-c_t\sqrt\delta
\left(\frac{\sqrt{D_0(X_-)}+\sqrt{D_0(X_+)}}2-
\sqrt{D_0(X_0)}\right)
+o(\sqrt\delta).
\end{multline*}
Observe that, the expression in parentheses is strictly negative by Lemma~\ref{lem:sloping-gap}.
Then since $c_t>0$, the right-hand side is strictly positive for all sufficiently
small $\delta>0$.  Fix such a value of $\delta$, and let
$$
J:=(-2\ell_0,2\ell_0), \qquad H:=H_\delta.
$$
Then (iii) holds. Finally, since $D_0\ge 1/2$ on $J$, 
$$\sup_{J} H_\delta\le h_c-\frac \delta 2 <h_c.$$
\end{proof}

\section{Construction of counterexamples}

\subsection{Construction of the uniformly convex domains}\label{sec:domains}

Let $H$, $J$, and $X_-,X_0,X_+$ be given by Lemma~\ref{lem:formal-midpoint}.  According to (ii) of Lemma~\ref{lem:formal-midpoint}, the values of $H$ on $J$ lie in the corresponding lower-branch half-width interval of $h_c$,  on which the one-dimensional
profiles are stable.  Choose a compact interval $J_0\subset\subset J$ that contains the
three points $X_-$, $X_0$, and $X_+$.

We also choose $h_*$ so close to $h_c$ that 
\begin{equation}\label{eq:h-star} 
\sup_{\overline{J_0}}H<h_*<h_c \quad  \text{ and } \quad
\left[\min_{X\in\overline{J_0}}H(X),\,h_*\right]
\subset (h_c-\delta_{\rm br},h_c).    
\end{equation}
This is possible since the values of $H$ were chosen sufficiently close to $h_c$ from below.
We will use the one-dimensional stable solution of half-width $h_*$ as a global supersolution.


By Lemma~\ref{lem:cap-completion}, applied with this $J_0$, there exists a bounded smooth uniformly convex set
$\mathcal O\subset\mathbb R^2$ such that
\begin{equation}\label{eq:O-inside-strip}
\overline{\mathcal O}\subset\{(X,\eta):|\eta|<h_*\},
\end{equation}
and
\begin{equation}\label{eq:O-exact-channel}
\mathcal O\cap (J_0\times\mathbb R)
=
\{(X,\eta):X\in J_0,\ |\eta|<H(X)\}.
\end{equation}

For $\epsilon>0$, define the horizontal dilation map
$$
T_\epsilon(X,\eta):=\left(\frac{X}{\epsilon},\eta\right),
$$
and set $\Omega_\epsilon:=T_\epsilon(\mathcal O).$
Since for every $\epsilon>0$, $T_\epsilon$ is a nonsingular affine map. Hence it preserves smoothness and  convexity. As $\partial \mathcal O$ is compact and has strictly positive curvature, then $\Omega_\epsilon$ also has strictly positive curvature. Thus $\Omega_\epsilon$ is uniformly convex, with a curvature lower bound that may depend on $\epsilon$. 
Moreover,
$$
\overline{\Omega_\epsilon}
\subset
\{(x_1,x_2): |x_2|<h_*\}.
$$
By \eqref{eq:O-exact-channel}, the central part of the scaled domain is exactly the slow channel:
$$
\Omega_\epsilon\cap\{(x_1,x_2):\epsilon x_1\in J_0\}
=\{(x_1,x_2):\epsilon x_1\in J_0,\ |x_2|<H(\epsilon x_1)\}.
$$
The three slow points from Lemma~\ref{lem:formal-midpoint} correspond to the
physical horizontal locations
$$
x_-:=\frac{X_-}{\epsilon}, \qquad
x_0:=\frac{X_0}{\epsilon}, \qquad
x_+:=\frac{X_+}{\epsilon}.
$$
These will be used in the final midpoint test for the actual superlevel set.

\subsection{The minimal solution and strict stability}\label{sec:solution-stability}

In this section, we construct the solution in the domains
$\Omega_\epsilon$ and establish a strict stability estimate, which will also be
employed subsequently in the slow-channel comparison later. We adopt the notation introduced in Section~\ref{sec:domains}. In particular, we fix $h_*\in(0,h_c)$ satisfying \eqref{eq:h-star} to be a stable lower-branch
half-width  and assume that
$$
\overline{\mathcal O}\subset \{(X,\eta): |\eta|<h_*\}.
$$

Let $\bar A\in(0,A_c)$ satisfy $h(\bar A)=h_*,$
and define
$$
\overline U(\xi):=\mathcal U_{h_*}(\xi)=U_{\bar A}(\xi),
\qquad -h_*\le \xi\le h_*.
$$
Thus
$$
-\overline U''=f(\overline U)\quad\text{in }(-h_*,h_*),
\qquad
\overline U(\pm h_*)=0,
\qquad
\overline U>0\quad\text{in }(-h_*,h_*).
$$
Since $h_*$ lies on the stable lower branch, the number
\begin{equation}\label{eq:defn-mu-star}
    \nu_*:=\lambda_1\left(-\frac{d^2}{d\xi^2}-f'(\overline U),\,(-h_*,h_*)\right)
\end{equation}
is strictly positive.

The following lemma is well-known. For the convenience of the reader, we give a quick proof here. 

\begin{proposition} \label{prop:minimal-stability}
Let $f\in C^3([0,\,\infty))$ with $f>0$, $f'\ge 0$ and $f''\ge 0$.  For every $\epsilon>0$ small, the problem
$$
\begin{cases}
-\Delta u=f(u) & \text{in }\Omega_\epsilon, \\
u>0 & \text{in }\Omega_\epsilon, \\
u=0 & \text{on }\partial\Omega_\epsilon,
\end{cases}
$$
has a minimal classical solution $u_\epsilon$. Moreover,
$$
0<u_\epsilon(x_1,x_2)\le \overline U(x_2)
\qquad\text{for every }(x_1,x_2)\in\Omega_\epsilon,
$$
and $u_\epsilon$ is strictly stable.  More precisely, for every
$\phi\in W_0^{1,2}(\Omega_\epsilon)$,
\begin{equation}\label{eq:strict-stability}
\int_{\Omega_\epsilon}
\left(|\nabla \phi|^2-f'(u_\epsilon)\phi^2\right)\,dx_1\,dx_2
\ge \nu_*\int_{\Omega_\epsilon}\phi^2\,dx_1\,dx_2.
\end{equation}
\end{proposition}

\begin{proof}
We first construct the minimal solution. Since
$\overline{\Omega_\epsilon}\subset\{(x_1,x_2): |x_2|<h_*\}$, 
the function
$(x_1,x_2)\mapsto \overline U(x_2)$ is nonnegative on $\partial\Omega_\epsilon$ and satisfies
$$
-\Delta \overline U(x_2)=-\overline U''(x_2)=f(\overline U(x_2))
\qquad\text{in }\Omega_\epsilon.
$$
Thus $\overline U$ is a supersolution.  Moreover, the zero function is a subsolution as  $-\Delta0=0\le f(0)$.

We define a monotone sequence as following. Let  $u_0\equiv0$ and, for $k\ge0$, let
$u_{k+1}$ solve
$$
\begin{cases}
-\Delta u_{k+1}=f(u_k) & \text{in }\Omega_\epsilon,\\
u_{k+1}=0 & \text{on }\partial\Omega_\epsilon.
\end{cases}
$$
Standard linear elliptic theory gives $u_{k+1}\in C^{2,\alpha}
(\overline{\Omega_\epsilon})$ for every $\alpha\in(0,1)$; see, for instance, \cite[Chapter 6]{GT2001}. Since
$f$ is nondecreasing, the maximum principle gives
$$
0=u_0\le u_1\le u_2\le\cdots\le \overline U
\qquad\text{in }\Omega_\epsilon.
$$
Indeed, the upper bound follows by comparing $u_{k+1}$ with
$\overline U$: If $u_k\le \overline U$, then
$$
-\Delta(\overline U-u_{k+1})=f(\overline U)-f(u_k)\ge0,
$$
and $\overline U-u_{k+1}\ge0$ on $\partial\Omega_\epsilon$. The monotonicity
of the sequence follows similarly from
$$
-\Delta(u_{k+1}-u_k)=f(u_k)-f(u_{k-1})\ge0.
$$

The sequence is uniformly bounded in $L^\infty(\Omega_\epsilon)$ as
$0\le u_k\le \overline U$.  Moreover $f(u_k)$ is uniformly bounded in
$L^\infty(\Omega_\epsilon)$.  The global $W^{2,q}$-estimates for the
Dirichlet problem in a smooth bounded domain, for any fixed $1<q<\infty$, give a
uniform $W^{2,q}$-bound for $u_{k+1}$.  Choosing $q>2$, the Sobolev
embedding gives a uniform H\"older bound for the sequence.  Hence the Arzel\'a--Ascoli
theorem gives locally uniform, and in fact uniform, convergence along subsequences.
Since the sequence is monotone, the whole sequence converges uniformly to a
function $u_\epsilon$.  Passing to the limit in the equation gives a weak
solution.  Then the elliptic regularity and the smoothness of $f$   imply
that $u_\epsilon$ is a classical solution; see, for instance,
\cite[Chapters 6 and 9]{GT2001}.  The bounds pass to the limit, so
$0\le u_\epsilon\le \overline U$.  In addition, the strong maximum principle gives
$u_\epsilon>0$ in $\Omega_\epsilon$.

The same iteration procedure proves minimality as well. Let $w$ be any positive classical
solution of the same Dirichlet problem. Since $u_0=0\le w$, induction gives
$u_k\le w$ for every $k$. Passing to the limit yields
$u_\epsilon\le w$. Therefore $u_\epsilon$ is the minimal positive solution.

It remains to prove \eqref{eq:strict-stability}. Since $f''\ge0$, the function
$f'$ is nondecreasing. Since $u_\epsilon\le \overline U$, we have
$$
f'(u_\epsilon(x_1,x_2))\le f'(\overline U(x_2))
\qquad\text{for every }(x_1,x_2)\in\Omega_\epsilon.
$$
It is enough to prove the estimate for smooth compactly supported
$\phi$ and then pass to all of $W_0^{1,2}$ by density. Extend $\phi$ by zero to
the full strip $\mathbb R\times(-h_*,h_*)$. For $\mathcal H^1$-almost every $x_1$, the slice
$x_2\mapsto \phi(x_1,x_2)$ belongs to $W_0^{1,2}(-h_*,h_*)$. Therefore the
one-dimensional spectral inequality gives
$$
\int_{-h_*}^{h_*}
\left( |\partial_{x_2}\phi(x_1,x_2)|^2-f'(\overline U(x_2))\phi(x_1,x_2)^2\right)\,dx_2
\ge
\nu_*\int_{-h_*}^{h_*}\phi(x_1,x_2)^2\,dx_2.
$$
Integrating this inequality in $x_1$ and using $|\nabla\phi|^2\ge
|\partial_{x_2}\phi|^2$, we conclude \eqref{eq:strict-stability}. The proposition follows.
\end{proof}

We also need that the local lower-branch profiles in the slowly varying channel are dominated by the global supersolution $\overline U$.

\begin{lemma}\label{lem:profile-domination}
Let $I_{\rm br} \subset (0, h_c)$ be a stable lower-branch half-width interval containing
$$
\{H(X):X\in\overline{J_0}\}\cup\{h_*\}.
$$
Assume that $h\in I_{\rm br}$, $0<h\le h_*$, and let $\mathcal U_h$ denote the corresponding stable
one-dimensional profile on $(-h,h)$. Then
$$
0\le\mathcal  U_h(\xi)\le \overline U(\xi)
\qquad\text{for every } |\xi|<h.
$$
In particular,
$$
0\le \mathcal U_{H(X)}(\eta)\le \overline U(\eta)
$$
whenever $X\in J_0$ and $|\eta|<H(X)$.
\end{lemma}

\begin{proof}
The restriction of $\overline U$ to $(-h,h)$ is a supersolution for the
one-dimensional Dirichlet problem on $(-h,h)$, since
$-\overline U''=f(\overline U)$
and $\overline U(\pm h)>0$ if $h<h_*$. If $h=h_*$, the result follows immediately. Assume therefore $h<h_*$. Monotone iteration in the proof of  Proposition~\ref{prop:minimal-stability} between the subsolution $0$ and the supersolution $\overline U$ gives a minimal positive
solution $m_h$ on $(-h,h)$, and
$$
0\le m_h\le \overline U.
$$
We claim that $m_h=\mathcal U_h$. Since $\mathcal U_h$ is strictly stable on the lower
branch, the operator
$$
L_h:=-\frac{d^2}{d\xi^2}-f'(\mathcal U_h)
$$
satisfies the maximum principle on $(-h,h)$. By minimality, $m_h\le \mathcal U_h$. Set
$$
W:=\mathcal U_h-m_h\ge0.
$$
Then $W=0$ at $\xi=\pm h$, and by convexity of $f$,
$$
L_h W
=f(\mathcal U_h)-f(m_h)-f'(\mathcal U_h)(\mathcal U_h-m_h)\le0.
$$
Since $\lambda_1(L_h, (-h,h))>0$, the maximum principle for $L_h$ gives $W\le 0$. Since also $W\ge0$, we have
$W\equiv0$. Thus $\mathcal U_h=m_h$, and consequently $\mathcal U_h\le\overline U$.
The final assertion follows by taking $h=H(X)$.
\end{proof}

\subsection{A maximum-principle approximation in the slow channel}\label{sec:c0-approx}

We now prove analytic approximation estimate needed for the counterexample.  

Let $I=(X_L,X_R)$ be an open interval such that
$$
\{X_-,\,X_0,\,X_+\}\subset I\subset\subset J_0.
$$
Later we will estimate on a smaller interval $I'\subset\subset I$ that still contains
$X_-$, $X_0$, and $X_+$.  For $X\in I$, define the slow channel as
$$
\mathcal Q_I:=\{(X,\eta): |\eta|<H(X)\},
$$
and, for $I'\subset\subset I$, define
$$
\mathcal Q_{I'}:=\{(X,\eta): X\in I',\ |\eta|<H(X)\}.
$$

Since $I\subset\subset J_0$, the identity \eqref{eq:O-inside-strip} implies that $(X/\epsilon, \eta)\in \Omega_\epsilon$ whenever $(X,\eta)\in \mathcal Q_I$.
For $(X,\eta)\in\mathcal Q_I$, set
\begin{equation}\label{eq:defn-v}
    v_\epsilon(X,\eta):=u_\epsilon\left(\frac{X}{\epsilon},\eta\right).
\end{equation}
Then $v_\epsilon$ satisfies
$$
-\epsilon^2(v_\epsilon)_{XX}-(v_\epsilon)_{\eta\eta}=f(v_\epsilon) \qquad\text{in }\mathcal Q_I.
$$
Define the slowly varying one-dimensional profile
$$
V(X,\eta):=\mathcal U_{H(X)}(\eta), \qquad (X,\eta)\in\mathcal Q_I,
$$
which satisfies
$$
-V_{\eta\eta}=f(V) \qquad\text{in }\mathcal Q_I,
$$
and
$$
V(X,\pm H(X))=0.
$$
Moreover, by Proposition~\ref{prop:minimal-stability} and
Lemma~\ref{lem:profile-domination},
$$
0\le v_\epsilon(X,\eta)\le \overline U(\eta),
\qquad
0\le V(X,\eta)\le \overline U(\eta)
\quad\text{for }\ (X,\eta)\in\mathcal Q_I.
$$

We first record the following comparison principle.  

\begin{lemma} \label{lem:slow-channel-maximum}
Let $c\in L^\infty(\mathcal Q_I)$ satisfy
$$
0\le c(X,\eta)\le f'(\overline U(\eta))
\qquad\text{for a.e. }(X,\eta)\in\mathcal Q_I.
$$
Let $w\in W^{1,2}(\mathcal Q_I)\cap C(\overline{\mathcal Q_I})$ satisfy
$$
-\epsilon^2 w_{XX}-w_{\eta\eta}-c(X,\eta)w\ge0
$$
in the weak sense in $\mathcal Q_I$, and assume that
$$
w\ge0\qquad\text{on }\partial\mathcal Q_I.
$$
Then
$$
w\ge0\qquad\text{in }\mathcal Q_I.
$$
\end{lemma}

\begin{proof}
Let $w_-:=\max\{-w,0\}.$
Since $w\ge0$ on $\partial\mathcal Q_I$, we have
$w_-\in W_0^{1,2}(\mathcal Q_I)$.  Testing the weak inequality with $w_-$
gives
$$
\int_{\mathcal Q_I}
\left(\epsilon^2 |(w_-)_X|^2+|(w_-)_\eta|^2-cw_-^2\right)\,dX d\eta
\le0.
$$
For almost every $X\in I$, extend the slice $\eta\mapsto w_-(X,\eta)$ by zero to
$(-h_*,h_*)$.  By the definition of $\nu_*$,
$$
\int_{-h_*}^{h_*}
\left(|\psi'(\xi)|^2-f'(\overline U(\xi))\psi(\xi)^2\right)\,d\xi
\ge
\nu_*\int_{-h_*}^{h_*}\psi(\xi)^2\,d\xi
$$
for every $\psi\in W_0^{1,2}((-h_*,h_*))$.  Applying this to the extended slice
and using $c\le f'(\overline U)$, we obtain
$$
\int_{\mathcal Q_I}
\left(|(w_-)_\eta|^2-cw_-^2\right)\,dX d\eta
\ge
\nu_*\int_{\mathcal Q_I}w_-^2\,dX d\eta.
$$
Consequently,
$$
\epsilon^2\int_{\mathcal Q_I}|(w_-)_X|^2\,dX d\eta
+ \nu_*\int_{\mathcal Q_I}w_-^2\,dX d\eta
\le0.
$$
Since $\nu_*>0$, this implies $w_-=0$.  Hence $w\ge0$ in
$\mathcal Q_I$.
\end{proof}

We  need the following auxiliary lemma. 

\begin{lemma}
\label{lem:slow-profile-xbounds}
Let $K=[h_1,h_2]\subset\subset(0,h_c)$ be a compact interval of half-widths contained in the stable lower branch. For $h\in K$, let $\mathcal U_h$ denote the lower-branch profile on $(-h,h)$, and define
$$
W(h,\sigma):=\mathcal U_h(h\sigma),
\qquad
(h,\sigma)\in K\times[-1,1].
$$
Then $W\in C^2(K\times[-1,1])$. 

Consequently, if $I\subset\subset J$, 
$H\in C^2(\overline I)$,   $H(\overline I)\subset K$, and 
$$
V(X,\eta):=\mathcal U_{H(X)}(\eta),
\qquad
(X,\eta)\in\mathcal Q_I:=\{(X,\eta):X\in I,\ |\eta|<H(X)\},
$$
then there exists a constant $C_V>0$, depending only on $f$, $K$, and
$\|H\|_{C^2(\overline I)}$, such that
$$
|V_{XX}(X,\eta)|\le C_V
\qquad\text{for every }(X,\eta)\in\mathcal Q_I.
$$
\end{lemma}

\begin{proof}
Let $A_-(h)$ denote the lower-branch inverse of the half-width map. Since
$K\subset\subset(0,h_c)$ is contained in the stable lower branch, the corresponding amplitudes form a compact interval
$$
A_-(K)\subset\subset(0,A_c).
$$
Let $\widetilde U(A,\xi)$ be the solution of the initial value problem
$$
-\widetilde U_{\xi\xi}=f(\widetilde U),
\qquad
\widetilde U(A,0)=A,
\qquad
\widetilde U_\xi(A,0)=0.
$$
By the smooth dependence of ordinary differential equations on initial data,
$\widetilde U$ is $C^2$ in $(A,\xi)$ on a neighborhood of the compact set
$$
\{(A,\xi):A\in A_-(K),\ |\xi|\le h(A)\}.
$$
Moreover, on the lower branch the map $A\mapsto h(A)$ is smooth and satisfies
$h'(A)>0$, so its inverse $h\mapsto A_-(h)$ is smooth on $K$. Therefore
$$
W(h,\sigma) = \mathcal U_h(h\sigma) 
= \widetilde U(A_-(h),h\sigma)  \in C^2(K\times[-1,1]).
$$

Now write $\sigma= {\eta}/{H(X)}.$
Then
$V(X,\eta)=W(H(X),\sigma).$
Define
$$
\Psi(h,\sigma):=W_h(h,\sigma)-\frac{\sigma}{h}W_\sigma(h,\sigma).
$$
Since $h\ge h_1>0$ on $K$ and $W\in C^2(K\times[-1,1])$, the function $\Psi$ is $C^1$ and bounded together with its first derivatives on $K\times[-1,1]$. Differentiating at a fixed  $\eta$ gives
$$
V_X(X,\eta)=H'(X)\Psi(H(X),\sigma),
$$
and
$$
V_{XX}(X,\eta)
=
H''(X)\Psi(H(X),\sigma)
+
(H'(X))^2
\left[
\Psi_h(H(X),\sigma)
- \frac{\sigma}{H(X)}\Psi_\sigma(H(X),\sigma)
\right].
$$
Observe that the right-hand side is uniformly bounded by a constant depending only on
$f$, $K$, and $\|H\|_{C^2(\overline I)}$. This proves the claim.
\end{proof}

\begin{proposition}\label{prop:c0-approx}
Let $I'\subset\subset I$.  There exist constants
$$
C>0, \qquad \epsilon_0>0,
\qquad \kappa>0,
$$
depending only on $f$, on $h_*$, on the compact stable branch interval
containing $H(\overline I)$, on $\|H\|_{C^2(\overline I)}$, and on the
intervals $I'\subset\subset I$, such that, for every $\epsilon\in(0,\epsilon_0]$,
$$
\sup_{(X,\eta)\in\mathcal Q_{I'}}
|v_\epsilon(X,\eta)-V(X,\eta)|
\le C\epsilon^2.
$$
\end{proposition}

\begin{proof}
Set $W_\epsilon:=v_\epsilon-V$
on $\mathcal Q_I$.  Since $V$ solves the one-dimensional equation in the
vertical variable,
$$
-\epsilon^2V_{XX}-V_{\eta\eta}-f(V)=-\epsilon^2V_{XX}.
$$
Subtracting the equation for $V$ from the equation for $v_\epsilon$, we
obtain
\begin{equation}\label{eq:difference-equation}
-\epsilon^2 (W_\epsilon)_{XX}-(W_\epsilon)_{\eta\eta}-c_\epsilon(X,\eta)W_\epsilon
=\epsilon^2 V_{XX}
\qquad\text{in }\mathcal Q_I,
\end{equation}
where
$$
c_\epsilon(X,\eta):=
\begin{cases}
\dfrac{f(v_\epsilon(X,\eta))-f(V(X,\eta))}{v_\epsilon(X,\eta)-V(X,\eta)},
& v_\epsilon(X,\eta)\ne V(X,\eta),\\
f'(V(X,\eta)),& v_\epsilon(X,\eta)=V(X,\eta).
\end{cases}
$$
Since $0\le v_\epsilon,V\le\overline U$ and $f'$ is nondecreasing,
\begin{equation}\label{eq:ceps-bound}
    0\le c_\epsilon(X,\eta)\le f'(\overline U(\eta))
\qquad\text{in }\mathcal Q_I.
\end{equation}


Choose a compact  interval
$K\subset\subset(0,h_c)$ such that $ H(\overline I)\subset K.$
By Lemma~\ref{lem:slow-profile-xbounds}, there exists a constant $C_V>0$,
depending only on $f$, $K$, and $\|H\|_{C^2(\overline I)}$, such that
\begin{equation}\label{eq:VXX-bound}
|V_{XX}(X,\eta)|\le C_V
\qquad\text{for every }(X,\eta)\in\mathcal Q_I.
\end{equation}

Now we start to construct barrier functions for $W_\epsilon$. 
Let $\varphi_*$ be the positive first Dirichlet eigenfunction of the operator
$$-\frac{d^2}{d\xi^2}-f'(\overline U)$$
on $(-h_*,h_*)$, normalized by $\max_{[-h_*,h_*]}\varphi_*=1. $
Thus, recalling \eqref{eq:defn-mu-star},
$$
-\varphi_*''-f'(\overline U)\varphi_*=\nu_*\varphi_*,
\qquad
\varphi_*>0\text{ in }(-h_*,h_*),
\qquad
\varphi_*(\pm h_*)=0.
$$
Since $\sup_{X\in\overline I}H(X)<h_*,$ then
there is a constant $m_\varphi>0$, depending only on $H|_{\overline I}$ and $h_*$, such that
$$
\varphi_*(\eta)\ge m_\varphi
\qquad\text{for every }(X,\eta)\in\mathcal Q_I.
$$
Using $c_\epsilon\le f'(\overline U)$, we have
$$
-\varphi_*''-c_\epsilon\varphi_*
= \nu_*\varphi_*+\bigl(f'(\overline U)-c_\epsilon\bigr)\varphi_*
\ge \nu_*\varphi_*.
$$
Choose $\kappa>0$ such that $0<\kappa^2<\nu_*.$
Define
$$
E_L(X):=e^{-\kappa(X-X_L)/\epsilon},
\qquad
E_R(X):=e^{-\kappa(X_R-X)/\epsilon}.
$$
Then
$$
\left(-\epsilon^2\partial_{XX}-\partial_{\eta\eta}-c_\epsilon\right)(E_L\varphi_*)
=E_L\left(-\kappa^2\varphi_*-\varphi_*''-c_\epsilon\varphi_*\right)
\ge 0,
$$
and the same inequality holds for $E_R\varphi_*$.

Let
\begin{equation}\label{eq:defn-M-star}
    M_*:=\|\overline U\|_{L^\infty((-h_*,h_*))}.
\end{equation}
Choose constants $C_0$ and $C_1$, depending only on the data listed in the
statement, so large that
$$
C_0\nu_* m_\varphi\ge C_V,
\qquad
C_1 m_\varphi\ge 2M_*.
$$
Set
$$
\mathcal B(X,\eta):=
C_0\epsilon^2\varphi_*(\eta)
+C_1\bigl(E_L(X)+E_R(X)\bigr)\varphi_*(\eta).
$$
The inequalities above imply
$$
\left(-\epsilon^2\partial_{XX}-\partial_{\eta\eta}-c_\epsilon\right) \mathcal B
\ge \epsilon^2 C_V
\ge \epsilon^2 \left| V_{XX}\right|
\qquad\text{in }\mathcal Q_I.
$$
Note that, on the top and bottom boundary arcs $\eta=\pm H(X)$, both $v_\epsilon$ and
$V$ vanish, and thus $W_\epsilon=0$.  On the artificial vertical sides $X=X_L$ and $X=X_R$,
the estimate $|W_\epsilon|\le2 M_*$ holds by \eqref{eq:defn-M-star}, while $\mathcal B\ge C_1m_\varphi\ge2 M_*$.
Therefore
$$
-\mathcal B\le W_\epsilon\le\mathcal B
\qquad\text{on }\partial\mathcal Q_I.
$$

Now the comparison principle in Lemma~\ref{lem:slow-channel-maximum} applies
due to \eqref{eq:ceps-bound}.
Applying that lemma to $\mathcal B-W_\epsilon$ and to
$\mathcal B+W_\epsilon$, and using \eqref{eq:difference-equation}, gives
$$
|W_\epsilon(X,\eta)|\le\mathcal B(X,\eta)
\qquad\text{in }\mathcal Q_I.
$$

Let $d_I:=\operatorname{dist}(I',\partial I)>0.$
For $X\in I'$,
$$
E_L(X)+E_R(X)
\le 2e^{-\kappa d_I/\epsilon}.
$$
After decreasing $\epsilon_0>0$ if necessary, we have
$$
e^{-\kappa d_I/\epsilon}\le \epsilon^2
\qquad\text{for }\ 0<\epsilon\le\epsilon_0.
$$
Consequently,
$$
|W_\epsilon(X,\eta)|\le C\epsilon^2
\qquad\text{for every }(X,\eta)\in\mathcal Q_{I'},
$$
where $C$ depends only on the stated data.  This proves the proposition.
\end{proof}

\subsection{Nonconvexity of the actual level}\label{sec:nonconvexity}

We now pass from the formal midpoint violation in
Lemma~\ref{lem:formal-midpoint} to the actual solution $u_\epsilon$.  The argument uses the $L^\infty$-estimate from Proposition~\ref{prop:c0-approx}. 

Choose open intervals $I'\subset\subset I\subset\subset J_0$ such that
$$
X_-,\ X_0,\ X_+\in I'.
$$
Recall that
$$
Y(X):=\mathcal Y_t(H(X)),
\qquad X\in I',
$$
is the upper height of the level $t$ of the slowly varying profile
$V(X,\eta)=\mathcal U_{H(X)}(\eta)$.  By Lemma~\ref{lem:formal-midpoint},
\begin{equation}\label{eq:formal-gap}
G:=\frac{Y(X_-)+Y(X_+)}2-Y(X_0)>0.
\end{equation}
Note that $G$ is independent of the parameter $\epsilon$ below. 

For $\epsilon>0$, let 
$$
E_\epsilon(t):=
\{(x_1,x_2)\in\Omega_\epsilon:u_\epsilon(x_1,x_2)>t\}.
$$
Then by pulling back with  horizontal dilation $T_\epsilon$, we define  superlevel set in the slow coordinates as
$$
S_\epsilon(t):=T_\epsilon^{-1}(E_\epsilon(t))=
\{(X,\eta): (X/\epsilon,\eta)\in\Omega_\epsilon,
\ u_\epsilon(X/\epsilon,\eta)>t\}.
$$
Namely, according to \eqref{eq:defn-v}
$$
S_\epsilon(t)=\{(X,\eta):v_\epsilon(X,\eta)>t\}.
$$
For every $X$ for which the vertical section is nonempty, define its \emph{upper height} by
$$
Y_\epsilon(X):=
\sup\{\eta\in\mathbb R:(X,\eta)\in S_\epsilon(t)\}.
$$

\begin{lemma} \label{lem:height-closeness}
There exist constants $C_h>0,$ and $\epsilon_h>0,$
depending only on
$$f,\, t,\, h_*,\, H|_{\overline I},\, I,\, I',$$
such that, for every $\epsilon\in(0,\epsilon_h]$ and each
$X\in\{X_-,X_0,X_+\}$, the number $Y_\epsilon(X)$ is well-defined and
$$
|Y_\epsilon(X)-Y(X)|\le C_h\epsilon^2.
$$
\end{lemma}

\begin{proof}
Set
$$
\mathcal K_{I'}:=H(\overline{I'}).
$$
By the choice of the branch neighborhood, $\mathcal K_{I'}$ is a compact subset
of $(h_c-\delta_{\rm br},h_c)$. In particular, by
\eqref{eq:level-below-amplitude},
$$
\min_{h\in\mathcal K_{I'}}\bigl(A_-(h)-t\bigr)>0.
$$
Thus, for every $h\in\mathcal K_{I'}$, the profile $\mathcal U_h$ satisfies
$$
\mathcal U_h(0)=A_-(h)>t, \qquad \mathcal U_h(h)=0<t.
$$
Since $\mathcal  U_h$ is strictly decreasing on $(0,h)$, there exists a unique point $\mathcal Y_t(h)\in(0,h)$ such that $\mathcal U_h(\mathcal Y_t(h))=t.$
Moreover, by the implicit function theorem and the smooth dependence of
$U_h$ on $h$, the map $h\mapsto \mathcal Y_t(h)$ is continuous on
$\mathcal K_{I'}$. Therefore, the upper height of the level set $t$ of $V$ 
$$Y(X)=\mathcal Y_t(H(X))$$ 
is continuous on $\overline{I'}$.

Next, since
$$
\mathcal U_h'(\mathcal  Y_t(h))<0
\qquad\text{for every }h\in\mathcal K_{I'},
$$
the compactness of $\mathcal K_{I'}$ gives
$$
\mu_t:= \min_{h\in\mathcal K_{I'}} \bigl(-\mathcal U_h'(\mathcal Y_t(h))\bigr)>0.
$$
Set
$$
d_t:=\frac12
\min_{h\in\mathcal K_{I'}}
\min\{\mathcal Y_t(h),\,h-\mathcal Y_t(h)\}>0.
$$
Choose $m_t>0$ so small that $4m_t\le \mu_t$. 
By the uniform continuity of
$\mathcal U_h'$, after decreasing $0<\rho_t<d_t$ if necessary, we may assume
$$
   -\mathcal U_h'(\eta)\ge 2m_t \quad \text{ 
whenever } \ h\in\mathcal K_{I'} \ \text{ and } \ |\eta-\mathcal Y_t(h)|\le \rho_t.
$$
Returning to $h=H(X)$, this gives
$$
-V_\eta(X,Y(X))\ge 2m_t, \qquad  Y(X)\ge 2d_t,  \qquad
H(X)-Y(X)\ge 2d_t
$$
for every $X\in\overline{I'}$. Finally, for $0\le s\le\rho_t$, integration of
the derivative bound gives
$$
V(X,Y(X)+s)-V(X,Y(X)) = \int_0^s V_\eta(X,Y(X)+r)\,dr  
\le -m_t s,
$$
and
$$
V(X,Y(X)-s)-V(X,Y(X)) = -\int_0^s V_\eta(X,Y(X)-r)\,dr
\ge m_t s.
$$
Since $V(X,Y(X))=t$, we obtain
\begin{equation}\label{eq:crossing-estimate}
\begin{aligned}
V(X,Y(X)+s)&\le t-m_t s,
\qquad 0\le s\le \rho_t,\\
V(X,Y(X)-s)&\ge t+m_t s,
\qquad 0\le s\le \rho_t.
\end{aligned}
\end{equation}

By Proposition~\ref{prop:c0-approx}, after decreasing $\epsilon_h$ if
necessary, there is a constant $C_0>0$, depending only on
$f,h_*,H|_{\overline I},I$, and $I'$, such that
\begin{equation}\label{eq:veps-V-diff}
    |v_\epsilon(X,\eta)-V(X,\eta)|\le C_0\epsilon^2
\qquad\text{for every }(X,\eta)\in\mathcal Q_{I'}
\end{equation}
whenever $0<\epsilon\le\epsilon_h$.  Choose
$K>\frac{C_0}{m_t}, $ and decrease $\epsilon_h$ further so that
\begin{equation}\label{eq:defn-K}
    K\epsilon^2\le \min\{\rho_t,d_t\} \qquad\text{for }\ 0<\epsilon\le\epsilon_h.
\end{equation}

Fix one of the three points $X\in\{X_-,X_0,X_+\}$.  We first prove an upper
bound for $Y_\epsilon(X)$. Let
$$
\eta\ge Y(X)+K\epsilon^2, \qquad \eta\le H(X).
$$
We have two cases. When
$$Y(X)+K\epsilon^2\le \eta\le Y(X)+\rho_t,$$ then the preceding crossing estimate \eqref{eq:crossing-estimate} gives
$$
V(X,\eta)\le t-m_tK\epsilon^2.
$$
If $\eta>Y(X)+\rho_t$, the monotonicity of $V(X,\cdot)$ on $(0,H(X))$ together with both \eqref{eq:crossing-estimate} and \eqref{eq:defn-K}  gives
$$
V(X,\eta) \le V(X,Y(X)+\rho_t)
\le t-m_t\rho_t \le t-m_tK\epsilon^2.
$$
Thus, in both of the cases, by \eqref{eq:veps-V-diff} and the choice on $K$,
$$
v_\epsilon(X,\eta)
\le t-(m_tK-C_0)\epsilon^2<t.
$$
It follows that no point with height at least $Y(X)+K\epsilon^2$ belongs to
$S_\epsilon(t)$ over this value of $X$.  Hence
$$
Y_\epsilon(X)
\le Y(X)+K\epsilon^2.
$$

For the lower bound, we first set $\eta=Y(X)-K\epsilon^2.$
Since $K\epsilon^2\le\rho_t$, the crossing estimate  \eqref{eq:crossing-estimate} gives
$$
V(X,\eta)
\ge t+m_tK\epsilon^2.
$$
Therefore, applying \eqref{eq:veps-V-diff} again yields
$$
v_\epsilon(X,\eta)
\ge t+(m_tK-C_0)\epsilon^2>t.
$$
The vertical section of $S_\epsilon(t)$ over this value of $X$ is therefore
nonempty, and
\begin{equation}\label{eq:height-lower-bound}
    Y_\epsilon(X)\ge Y(X)-K\epsilon^2.
\end{equation}
Taking $C_h:=K$ proves the lemma.
\end{proof}

\begin{proposition} \label{prop:actual-midpoint}
For all sufficiently small $\epsilon>0$, the superlevel set
$$
\{(x_1,x_2)\in\Omega_\epsilon:u_\epsilon(x_1,x_2)>t\}
$$
is not convex.
\end{proposition}

\begin{proof}
Apply Lemma~\ref{lem:height-closeness} at the three slow points
$X_-,X_0,X_+$.  For $0<\epsilon\le\epsilon_h$,
$$
\begin{aligned}
\frac{Y_\epsilon(X_-)+Y_\epsilon(X_+)}2-Y_\epsilon(X_0)
&\ge
\frac{Y(X_-)+Y(X_+)}2-Y(X_0)-2C_h\epsilon^2\\
&=G-2C_h\epsilon^2,
\end{aligned}
$$
where $G>0$ is the formal gap in \eqref{eq:formal-gap}.  Taking
$\epsilon>0$ so small that
$4C_h\epsilon^2<G,$
we conclude that
\begin{equation}\label{eq:actual-gap}
\frac{Y_\epsilon(X_-)+Y_\epsilon(X_+)}2>Y_\epsilon(X_0).
\end{equation}

Now assume, for contradiction, that the superlevel set
$$
E_\epsilon(t):=
\{(x_1,x_2)\in\Omega_\epsilon:u_\epsilon(x_1,x_2)>t\}
$$
is convex.  As the affine image of a convex set is convex, so its slow-coordinate image
$$
S_\epsilon(t)=\{(X,\eta):v_\epsilon(X,\eta)>t\}.
$$
is also convex; recall \eqref{eq:defn-v}.  By Lemma~\ref{lem:convex-height}, the upper height function $Y_\epsilon$ of $S_\epsilon(t)$ is concave on its horizontal projection.  The preceding height lower bound \eqref{eq:height-lower-bound} shows that 
$$ (X, Y(X)-K\epsilon^2) \in S_\epsilon(t)\qquad \text{ for } \ X\in \{X_-,\, X_0,\,X_+\};$$
hence
$X_-$, $X_0$, and $X_+$ belong to this
projection.  Since $X_0=(X_-+X_+)/2$, the concavity of $Y_\epsilon$  gives
$$
Y_\epsilon(X_0)
\ge
\frac{Y_\epsilon(X_-)+Y_\epsilon(X_+)}2,
$$
which contradicts \eqref{eq:actual-gap}.  Hence $E_\epsilon(t)$ is not convex.
\end{proof}

\begin{proof}[Proof of Theorem~\ref{thm:main}]
Let $f$ be fold-admissible.  Choose the level $t\in(0,A_c)$, the slow width
$H$, and the three slow points $X_-,X_0,X_+$ as in
Lemma~\ref{lem:formal-midpoint}.  Construct the smooth bounded uniformly convex
domains $\Omega_\epsilon$ as in Section~\ref{sec:domains}.  For each
$\epsilon>0$, Proposition~\ref{prop:minimal-stability} gives a minimal
classical solution $u_\epsilon$ in $\Omega_\epsilon$, and this solution is
strictly stable in the sense of \eqref{eq:strict-stability}.  By
Proposition~\ref{prop:actual-midpoint}, for all sufficiently small
$\epsilon>0$, the superlevel set
$$
\{(x_1,x_2)\in\Omega_\epsilon:u_\epsilon(x_1,x_2)>t\}
$$
is not convex.  Fix such an $\epsilon$, and define
$$
\Omega := \Omega_\epsilon,\qquad u := u_\epsilon,\qquad t_0 := t.
$$
Then $\Omega$, $u$, and $t_0$ satisfy all the properties stated in the main theorem.
\end{proof}

\begin{proof}[Proof of Corollary~\ref{cor:model}]
For $f(s)=e^s$ and $f(s)=(a+s)^p$, with $a>0$ and $p>1$,
Proposition~\ref{prop:exp-fold}  and Proposition~\ref{prop:power-fold} prove  their fold-admissibility, respectively.  Applying Theorem~\ref{thm:main} gives the desired counterexample.  This proves the corollary.
\end{proof}

Recall Theorem~\ref{thm:existence}, which establishes the existence of stable solutions for the parameterized problems. We conclude this section with the following corollary concerning such parameterized problems.

\begin{corollary} \label{cor:parameter-version}
Let $f$ be fold-admissible and let $\lambda>0$.  Then there exist a bounded
smooth uniformly convex planar domain $\Omega_\lambda$, a number $t_0>0$, and
a minimal solution $u_\lambda$ of
$$
\begin{cases}
-\Delta u_\lambda=\lambda f(u_\lambda) & \text{in }\Omega_\lambda,\\
u_\lambda>0 & \text{in }\Omega_\lambda,\\
u_\lambda=0 & \text{on }\partial\Omega_\lambda,
\end{cases}
$$
such that $u_\lambda$ is strictly stable and $\{u_\lambda>t_0\}$ is not
convex.  In particular, this applies to $f(s)=e^s$ and to
$f(s)=(a+s)^p$ with $a>0$ and $p>1$.
\end{corollary}

\begin{proof}
By Theorem~\ref{thm:main}, we choose a bounded smooth uniformly convex domain
$\Omega$, a strictly stable minimal solution $u$, and a level $t_0>0$ such
that
$$
-\Delta u=f(u)\quad\text{in }\Omega,
\qquad u=0\quad\text{on }\partial\Omega,
$$
and $\{u>t_0\}$ is not convex.  Define
$$
\Omega_\lambda:=\lambda^{-1/2}\Omega =\{x\in\mathbb R^2:\sqrt\lambda x\in\Omega\},
$$
and
$$
u_\lambda(x):=u(\sqrt\lambda x), \qquad x\in\Omega_\lambda.
$$
Then $\Omega_\lambda$ is bounded, smooth, and uniformly convex.  A direct
calculation gives
$$
-\Delta u_\lambda(x) =-\lambda (\Delta u)(\sqrt\lambda x)
=\lambda f(u(\sqrt\lambda x)) =\lambda f(u_\lambda(x)).
$$
Additionally, the boundary condition and positivity follows immediately.

Minimality is also preserved by the same scaling.  Indeed, if $w$ is any positive classical solution of
$$
-\Delta w=\lambda f(w)
\quad\text{in }\Omega_\lambda,
\qquad w=0\quad\text{on }\partial\Omega_\lambda,
$$
then
$$
W(X):=w(X/\sqrt\lambda), \qquad X\in\Omega,
$$
solves
$$
-\Delta W=f(W) \quad\text{in }\Omega,
\qquad W=0\quad\text{on }\partial\Omega.
$$
Since $u$ is minimal for the unscaled problem, $u\le W$ in $\Omega$.  Hence, $u_\lambda\le w$ in $\Omega_\lambda$  and $u_\lambda$ is minimal.

The strict stability also scales accordingly.  If $\phi\in C_c^1(\Omega_\lambda)$ and
$\Phi(X):=\phi(X/\sqrt\lambda)$, then the change of variables
$X=\sqrt\lambda x$ gives
$$
\int_{\Omega_\lambda}
\left(|\nabla\phi|^2-\lambda f'(u_\lambda)\phi^2\right)\,dx
= \int_\Omega \left(|\nabla\Phi|^2-f'(u)\Phi^2\right)\,dX.
$$
We note that the right-hand side is positive for every nonzero $\Phi\in W_0^{1,2}(\Omega)$, as $u$ is strictly stable.  Thus $u_\lambda$ is strictly stable. Finally,
$$
\{u_\lambda>t_0\}=\lambda^{-1/2}\{u>t_0\}.
$$
Since dilations preserve both convexity and nonconvexity, the superlevel set
$\{u_\lambda>t_0\}$ is not convex.  The final conclusion follows from Corollary~\ref{cor:model}.
\end{proof}

\end{document}